\documentclass{amsart}

\usepackage{amsmath,amssymb,amsthm,url}
\usepackage{hyperref,url}
\usepackage{tikz}

\numberwithin{equation}{section}

\theoremstyle{plain}
\newtheorem{theorem}{Theorem}[section]
\newtheorem{lemma}[theorem]{Lemma}
\newtheorem{proposition}[theorem]{Proposition}
\newtheorem{corollary}[theorem]{Corollary}

\theoremstyle{definition}

\newtheorem{notation}[theorem]{Notation}
\newtheorem{remark}[theorem]{Remark}

\usetikzlibrary{matrix}
\usetikzlibrary{positioning}
\usetikzlibrary{calc}

\newcommand{\natplus}{\mathbb{N}^+}
\newcommand{\flo}{\underline{f}}
\newcommand{\fvi}{f_\mathrm{vi}}
\newcommand{\mvi}{m_\mathrm{vi}}
\newcommand{\svi}{s_\mathrm{vi}}
\newcommand{\fpc}{f_\mathrm{pc}}

\newcommand{\F}{\mathcal{F}}
\renewcommand{\bot}{\hat{0}}
\renewcommand{\top}{\hat{1}}

\title[Exponential lower bounds of lattice counts]{Exponential lower bounds of lattice counts \\ by vertical sum and 2-sum}

\begin{document}

\author[J. Kohonen]{Jukka Kohonen}
\address{Department of Computer Science\\ University of Helsinki, Finland}
\email{jukka.kohonen@iki.fi}

\begin{abstract}
  We consider the problem of finding lower bounds on the number of
  unlabeled $n$-element lattices in some lattice family.  We show that
  if the family is closed under vertical sum, exponential lower bounds
  can be obtained from vertical sums of small lattices whose numbers
  are known.  We demonstrate this approach by establishing that the
  number of modular lattices is at least $2.2726^n$ for $n$ large
  enough.

  We also present an analogous method for finding lower bounds on the
  number of vertically \emph{indecomposable} lattices in some family.
  For this purpose we define a new kind of sum, the \emph{vertical
    2-sum}, which combines lattices at two common elements.  As an
  application we prove that the numbers of vertically indecomposable
  modular and semimodular lattices are at least $2.1562^n$ and
  $2.6797^n$ for $n$ large enough.
\end{abstract}

\maketitle

\section{Introduction}
\label{sec:intro}

One of the most elementary questions regarding a family of
combinatorial objects is: how many are they?  For various lattice
families this question has been approached in two ways.  Small
lattices have been generated by computation, and counted exactly.
Numbers of large lattices have been lower and upper bounded by
assorted methods.

The purpose of this note is to demonstrate that in some lattice
families, useful exponential lower bounds are obtained from vertical
compositions of small lattices, which have been counted by
computation.  By exponential we mean $c^n$, where $c$ is a constant
and $n$ is the number of elements.

We consider two kinds of vertical composition.  First we show how the
ordinary \emph{vertical sum} leads to exponential lower bounds in
families that are closed under vertical sum.  As an application, we
establish that the number of unlabeled modular lattices is at least
$2.2726^n$ for $n$ large enough.  This improves upon the previous
bound $2^{n-3}$ by Jipsen and Lawless~\cite{jipsen2015}.  Our bound is
derived from the counts of vertically indecomposable modular lattices
of $n \le 30$ elements, computed by the author in~\cite{kohonen2017}.
Further computations are likely to yield improved lower bounds.

Secondly we target the numbers of vertically \emph{indecomposable}
lattices, which may be more interesting.  To this end we define the
\emph{vertical 2-sum}, and show that it yields exponential lower
bounds on vertically indecomposable lattices.  As an application, we
establish that the numbers of unlabeled, vertically indecomposable
modular and semimodular lattices are at least $2.1562^n$ and
$2.6797^n$ for $n$ large enough.

\section{Vertical sum}
\label{sec:vertical}

All lattices in this work are finite, nonempty and unlabeled.  If $L$
and $U$ are lattices, their \emph{vertical sum} $L+U$ is defined by
identifying the top element of~$L$ with the bottom element of~$U$.
The vertical sum is associative, and the vertical sum of several
lattices is defined in the obvious way.  In fact, lattices with
vertical sum are a monoid, with the singleton lattice as its neutral
element.  For completeness we define the empty vertical sum to be the
singleton.

A~lattice $X$ is \emph{vertically decomposable} if it contains a
\emph{knot}, that is, an element distinct from top and bottom and
comparable to all elements.  One can then decompose $X$ at the knot
into two non-singleton lattices $L$ and $U$, whose vertical sum
is~$X$.  A~lattice that has no knot is \emph{vertically
  indecomposable}, or a \emph{vi-lattice}.  It is well
known~\cite{erne2002} that every finite nonempty lattice has a unique
vertical decomposition, that is, a representation as a vertical sum of
non-singleton vi-lattices.\footnote{The literature is varied on
  whether the singleton lattice is defined as vertically
  indecomposable.  In any case it needs some special treatment to
  ensure that vertical decompositions are unique.  Ern\'e \emph{et
    al.}~\cite{erne2002} define the singleton to be vertically
  \emph{decomposable}.  Although this feels odd, it is analogous to
  the now standard practice of excluding $1$ from primes to make prime
  factorization unique.  Some other authors tacitly include the
  singleton among vi-lattices~\cite{heitzig2002,jipsen2015}.  We
  define it as a vi-lattice but exclude it explicitly when necessary.}

\begin{notation}
  We will generally write $f(n)$ for the number of $n$-element
  lattices in some family, and $\fvi(n)$ for the corresponding number
  of vi-lattices.  For the numbers of modular lattices and modular
  vi-lattices, we write $m(n)$ and $\mvi(n)$.  For semimodulars we
  write $s(n)$ and $\svi(n)$.
\end{notation}

Vertical sum and decomposition have become standard tools in lattice
counting, due to the following observation (cf.~\cite[Equation
1]{heitzig2002}).

\begin{lemma}
  Let $\F$ be a lattice family that is closed under vertical sum and
  vertical decomposition, and contains the singleton lattice.  Let
  $f(n)$ and $\fvi(n)$ be the numbers of $n$-element lattices and
  vi-lattices in $\F$, respectively.  Then $f$~and $\fvi$ are related
  by
  \begin{equation}
    f(n) = \sum_{k=2}^n \fvi(k) \, f(n-k+1), \qquad \text{for $n \ge 2$.}
    \label{eq:recursion}
  \end{equation}
  \label{lemma:recursion}
\end{lemma}

\begin{proof}
  Let $n \ge 2$.  Each $n$-element lattice $X \in \F$ can be uniquely
  represented as a vertical sum $X = L+U$, where $L$ is vertically
  indecomposable and $|L| \ge 2$.  Because $\F$ is closed under
  vertical decomposition, we have $L,U \in \F$.  Note that if $X$ is
  vertically indecomposable, we still have $X = L+U$, where $L=X$,
  and $U$ is the singleton.

  The sum \ref{eq:recursion} counts such vertical sums, with $k$
  iterating over the possible cardinalities of $L$.  For each value of
  $k$, there are $\fvi(k)$ choices for $L$, and $f(n-k+1)$ choices for
  $U$.  In the boundary case $k=n$ we have $f(n-k+1)=f(1)=1$ as we
  assumed that the singleton is in~$\F$.  Also, each such vertical sum
  gives a lattice in $\F$, because $\F$ is closed under vertical sum.
\end{proof}

Modular, semimodular, distributive, and graded lattices are examples
of families where Lemma~\ref{lemma:recursion} applies.  It is well
known that \eqref{eq:recursion} can be used to reduce the workload
when counting small lattices by exhaustive generation.  The idea is to
generate only the vi-lattices in $\F$ up to some maximum size~$N$,
thus obtaining the values $\fvi(2),\ldots,\fvi(N)$, and then to
calculate $f(2),\ldots,f(N)$ by the recurrence.  This method has been
used with various lattice
families~\cite{erne2002,gebhardt2016,heitzig2002,jipsen2015,kohonen2017}.

We must point out that Lemma~\ref{lemma:recursion} requires the family
to be closed \emph{both} under vertical sum \emph{and} under vertical
decomposition.  Being closed under vertical sum is not enough: as a
counterexample, consider the family ``graded lattices of even rank''.
It contains lattices such as the 5-element chain that are not
accounted for by the sum \eqref{eq:recursion}, as their vi-components
fall out of the family.  But being closed under vertical sum suffices
for the \emph{inequality}
\begin{equation}
  f(n) \ge \sum_{k=2}^n \fvi(k) \, f(n-k+1), \qquad \text{for $n \ge 2$,}
  \label{eq:recursionineq}
\end{equation}
which is enough for proving lower bounds on $f(n)$.

We now proceed to demonstrate that besides exact counting of small
lattices, vertical sums are also useful for exponential lower bounds
on $f(n)$, that is, bounds of the form
\[f(n) \ge c^n
\]
with some constant~$c$.  The simplest way is to take vertical sums of
constant-size lattices; we begin with this method to illustrate its
ease.  (But we will prove stronger bounds later.)

\begin{theorem}
  Let $\F$ be a lattice family that is closed under vertical sum, and
  contains the $2$-element chain.  Let $f(n)$ be the number of
  $n$-element lattices in $\F$, and $N \ge 2$ an integer constant.
  Then $f(n) \ge \Omega(c^n)$, where $c = f(N)^{1/(N-1)}$.
  \label{thm:easy}
\end{theorem}

\begin{proof}
  Let $c$ be as stated.  We first prove the case when $n = (N-1)h+1$,
  where $h \ge 1$ is an integer.  Consider $h$-tuples
  $(L_1,L_2,\ldots,L_h)$ of $N$-element lattices in~$\F$.  There are
  $f(N)^h$ such tuples; each gives rise to a vertical sum
  \[
  L_1 + L_2 + \cdots + L_h = X,
  \]
  which is a lattice of $n = (N-1)h+1$ elements and belongs to~$\F$ by
  assumption.  Different tuples give rise to different lattices,
  because for each such $X$ there is only one way of breaking $X$ into
  a vertical sum of $h$ components of $N$ elements each.  Thus the
  number of $n$-element lattices in $\F$ is lower bounded by the
  number of the tuples:
  \[
  f(n) \ge f(N)^h = f(N)^{(n-1)/(N-1)} = c^{n-1}.
  \]

  For arbitrary $n \ge N$ we round $n$ down to the nearest value where
  the previous case applies.  More precisely, let $n'$ be the largest
  integer of the form $n' = (N-1)h+1$ such that $n' \le n$ and $h \ge
  1$ is an integer.  Note that $n' \ge n-N+2$.  Because $\F$ contains
  the 2-element chain, $f$ is nondecreasing (any $n$-element lattice
  can be extended to $n+1$ elements by adding the 2-element chain on
  top).  Thus
  \[
  f(n) \ge f(n') \ge c^{n'-1} \ge bc^n,
  \]
  where $b = c^{1-N}$ is a constant.  This holds for all $n \ge N$, so
  $f(n) \ge \Omega(c^n)$.
\end{proof}

\begin{corollary}
  $m(n) \ge \Omega(2.1332^n)$.
\end{corollary}
\begin{proof}
  Apply Theorem~\ref{thm:easy} with $N=30$ and
  $m(30)=3\,485\,707\,007$ \cite{kohonen2017}.
\end{proof}

\begin{corollary}
  $s(n) \ge \Omega(2.5080^n)$.
\end{corollary}
\begin{proof}
  Apply Theorem~\ref{thm:easy} with $N=25$ and
  $s(25)=3\,838\,581\,926$ \cite{kohonen2017}.
\end{proof}

Stronger lower bounds are obtained by applying the
recurrence~\eqref{eq:recursionineq}.  Let $N \ge 2$ be a constant, and
suppose that $\fvi(1),\fvi(2),\ldots,\fvi(N)$ are known.  Then we can
lower bound $f(n)$ by a constant-coefficient recursive sequence as
follows.

\begin{theorem}
  Let $\F$ be a lattice family closed under vertical sum, containing
  the singleton.  Let $f(n)$ and $\fvi(n)$ be the numbers of $n$-element
  lattices and vi-lattices in~$\F$.
  Let $\flo \colon \natplus \to \natplus$ be the sequence defined
  by $\flo(1)=1$,
  \begin{equation}
    \flo(n) = \sum_{k=2}^n \fvi(k) \, \flo(n-k+1)
    \label{eq:initial}
  \end{equation}
  when $n=2,3,\ldots,N$, and
  \begin{equation}
    \flo(n) = \sum_{k=2}^N \fvi(k) \, \flo(n-k+1)
    \label{eq:truncated}
  \end{equation}
  when $n \ge N+1$.  Then $f(n) \ge \flo(n)$ for all $n \ge 1$.
  Furthermore, the infinite sequence $\flo$ is determined by
  $\fvi(1),\fvi(2),\ldots,\fvi(N)$ through a homogeneous linear
  recurrence relation of order $N-1$.
  \label{thm:homo}
\end{theorem}

\begin{proof}
  For $n=1$, we have $f(n) = \flo(n) = 1$.  For $n=2,3,\ldots,N$, the
  claim $f(n) \ge \flo(n)$ holds by~\eqref{eq:recursionineq}.  For
  $n \ge N+1$ it holds because the right hand side
  of~\eqref{eq:truncated} is a truncated form of the right hand side
  of~\eqref{eq:recursionineq}.

  Let us substitute $i=k-1$ and write $\fvi(i+1)=a_i$ to emphasize
  that these are known constants.  The recurrence~\ref{eq:truncated}
  now becomes
  \begin{equation}
    \flo(n) = \sum_{i=1}^{N-1} a_i \, \flo(n-i).
    \label{eq:homo}
  \end{equation}
  This is a homogeneous linear recurrence relation of order $N-1$ with
  constant coefficients.  The values of
  $\fvi(1),\fvi(2),\ldots,\fvi(N)$ determine both the initial values
  of $\flo(n)$ up to $n=N$, and the coefficients of the recurrence.
  Thus they also determine the whole sequence $\flo$.
\end{proof}

The remaining task is to find an exponential lower bound for
$\flo(n)$.  A~standard method for solving recurrence relations such as
\eqref{eq:homo} begins by finding the roots of the auxiliary equation
\begin{equation}
  x^{N-1} = \sum_{i=1}^{N-1} a_i \, x^{N-1-i}.
  \label{eq:aux}
\end{equation}
We refer to~\cite[\S 7.7]{allenby2010} for details.  Let $r$ be the
root of \eqref{eq:aux} whose absolute value is the largest.  If $r$ is
a single root, then a solution to the recurrence \eqref{eq:homo} is of
the form $br^n + o(r^n)$, where $b$ is a constant.  Generally we will
have to find the roots numerically.  In order to obtain a rigorous
lower bound, one which is not subject to floating point errors, we
will choose $c$ slightly smaller than~$r$, and then prove directly
that $\flo(n) \ge c^n$ for $n$ large enough.

\begin{proposition}
  $m(n) \ge 2.2726^n$ for all $n$ large enough.
  \label{prop:m}
\end{proposition}

\begin{proof}
  Let $N=30$, and define $\flo$ as in Theorem~\ref{thm:homo}, with the
  values of $\fvi(1),\fvi(2),\ldots,\fvi(30)$ taken from the ``modular
  vi'' column of \cite[Table~1]{kohonen2017}.  By
  Theorem~\ref{thm:homo} we have $m(n) \ge \flo(n)$ for all $n \ge 1$.

  The auxiliary equation~\eqref{eq:aux} is now
  \begin{multline*}
    x^{29} = x^{28} +x^{26} +x^{25} +2x^{24} +3x^{23} +7x^{22} +12x^{21} +28x^{20} +54x^{19}
    +127x^{18} \\+266x^{17} +614x^{16} +1356x^{15} +3134x^{14} +7091x^{13} +16482x^{12} 
    +37929x^{11}\\ +88622x^{10} +206295x^{9} +484445x^{8} +1136897x^{7} +2682451x^{6} 
    +6333249x^{5} \\ +15005945x^{4} 
    +35595805x^{3} +84649515x^{2} +201560350x +480845007.
    \label{eq:auxmod}
  \end{multline*}
  Numerically we find that the root with the largest absolute value is
  a single real root $r \approx 2.272651$.  For a lower bound, we take
  $c = 2.2726$ and claim that $\flo(n) \ge c^n$ for $n \ge 150\,000$.
  We prove this by induction.  Applying \eqref{eq:truncated}
  recursively, we see that the claim holds for $150\,000 \le n \le
  150\,028$, which serves as the base case.  We then observe that if
  $\flo(k) \ge 2.2726^{k}$ for $29$ consecutive values from $k=n-29$
  to $k=n-1$, then by applying these inequalities in
  \eqref{eq:truncated} we have $\flo(n) \ge 2.2726^{n}$.  This
  completes the induction.
\end{proof}

The new bound improves upon the bound $m(n) \ge 2^{n-3}$ by Jipsen and
Lawless~\cite{jipsen2015}, but still falls short of the empirical
growth rate.  The ratios $\frac{m(n)}{m(n-1)}$ and
$\frac{\mvi(n)}{\mvi(n-1)}$ for $n \le 30$ look like $m(n)$ and
$\mvi(n)$ are growing roughly as $2.4^n$ \cite{kohonen2017}.  If the
values of $\mvi$ are computed further, Proposition~\ref{prop:m} is
likely to yield improved lower bounds.  For example, if further
computations reveal that $\mvi(31) \ge 2.35 \mvi(30)$, which seems
likely, then the constant $c$ in our lower bound will increase by
about $0.0060$.

For semimodular lattices, no previous lower bound seems to be known,
other than that of modulars. Using the values of $\svi(n)$ for $n \le
25$ from~\cite{kohonen2017}, Theorem~\ref{thm:homo} yields a lower
bound $s(n) \ge 2.6459^n$ for $n$ large enough.  We omit the details
because the bound is superseded by a stronger lower bound on
semimodular vi-lattices in the next section.  However, even the
stronger bound is only exponential.  We note that the ratios of the
consecutive values $\svi(22), \svi(23), \svi(24),$ and $\svi(25)$ are
$3.5082$, $3.5579$ and $3.6057$~\cite{kohonen2017}.  Since the ratios
are steadily increasing, we suspect that the growth of $s(n)$ may be
faster than exponential.

We can try applying Theorem~\ref{thm:homo} to other lattice families.
For distributive lattices, using the data for $n \le 49$ by Ern\'e
\emph{et al.}~\cite{erne2002}, we get a lower bound of $1.8388^n$,
which does not improve upon their results.  For graded lattices, using
the data for $n \le 21$ by the author~\cite{kohonen2017}, we get a
lower bound of $3.4015^n$, but this is not really useful, because it
is already known that their growth is faster than exponential.  From
Klotz and Lucht~\cite{klotz1971} and Kleitman and
Winston~\cite{kleitman1980} we have lower and upper bounds of the form
$c^{n^{3/2}+o(n^{3/2})}$ both for graded lattices and for all
lattices.

Let us conclude this section with a brief qualitative comparison.
From subset relations between families, we have
\[
d(n) \le m(n) \le s(n) \le g(n) \le \ell(n),
\]
where $d(n)$, $g(n)$, and $\ell(n)$ are the numbers of distributive
lattices, graded lattices, and all lattices of $n$ elements.  For
$d(n)$, exponential lower \emph{and} upper bounds are
known~\cite{erne2002}.  For $m(n)$ we have an exponential lower bound,
and the empirical growth seems exponential, but an exponential upper
bound is lacking; the only known upper bound on (semi)modulars seems
to be that of all lattices~\cite{jipsen2015}.  For $s(n)$ we have an
exponential lower bound, but empirically the growth seems faster.  The
growths of $g(n)$ and $\ell(n)$ are \emph{known} to be faster than
exponential.  It remains a topic of further study to better separate
the growth rates of different lattice families.
  
\section{Vertical 2-sum}
\label{sec:vertical2}

We now turn our attention to the numbers of vertically indecomposable
lattices.  Our method is similar to the previous section: arbitrarily
large lattices are constructed from smaller lattices, whose number is
known.

\tikzset{point/.style={circle,inner sep=0,minimum size=5 pt,fill=black}}
\tikzset{invisible/.style={circle,inner sep=0,minimum size=5 pt,fill=none}}

\begin{figure}[b]
  \centering
  \begin{tikzpicture}[point/.style={circle,minimum size=5pt,inner sep=0,fill=black}]
    \matrix(L)[matrix of math nodes, column sep=0.33cm, row sep=0.7cm]{
      \node[invisible](x){}; \\
      & \node[point](a){} ; \\
      \node[point](b){}; && \node[point](c){}; \\
      & \node[point](d){}; & \node[point](e){}; & \node[point](f){}; \\
      && \node[point](g){} ; \\      
    };
    \draw (a)--(b);
    \draw (a)--(c);
    \draw (b)--(d);
    \draw (c)--(d);
    \draw (c)--(e);
    \draw (c)--(f);
    \draw (d)--(g);
    \draw (e)--(g);
    \draw (f)--(g);
    \draw[dashed] ($(b)+(-0.5,0)$)--($(c)+(0.5,0)$);
  \end{tikzpicture}
  \hspace{1.2cm}
  \begin{tikzpicture}
    \matrix(U)[matrix of math nodes, column sep=0.33cm, row sep=0.7cm]{
      && \node[point](1){} ; \\
      \node[point](2){}; && \node[point](3){}; && \node[point](4){}; \\
      & \node[point](5){}; && \node[point](6){}; \\
      && \node[point](7){} ; \\
      \node[invisible](xx){}; \\
    };
    \draw (1)--(2);
    \draw (1)--(3);
    \draw (1)--(4);
    \draw (2)--(5);
    \draw (3)--(5);
    \draw (3)--(6);
    \draw (4)--(6);
    \draw (5)--(7);
    \draw (6)--(7);
    \draw[dashed] ($(5)+(-0.5,0)$)--($(6)+(0.5,0)$);
  \end{tikzpicture}
  \hspace{1.2cm}
  \begin{tikzpicture}[point/.style={circle,minimum size=5pt,inner sep=0,fill=black}]
    \matrix(V)[matrix of math nodes, column sep=0.33cm, row sep=0.7cm]{
      && \node[point](1){} ; \\
      \node[point](2){}; && \node[point](3){}; && \node[point](4){}; \\
      & \node[point](b){}; && \node[point](c){}; \\
      && \node[point](d){}; & \node[point](e){}; & \node[point](f){}; \\
      &&& \node[point](g){} ; \\      
    };
    \draw (1)--(2);
    \draw (1)--(3);
    \draw (1)--(4);
    \draw (2)--(b);
    \draw (3)--(b);
    \draw (3)--(c);
    \draw (4)--(c);
    \draw (b)--(d);
    \draw (c)--(d);
    \draw (c)--(e);
    \draw (c)--(f);
    \draw (d)--(g);
    \draw (e)--(g);
    \draw (f)--(g);
    \draw[dashed] ($(b)+(-0.5,0)$)--($(c)+(0.5,0)$);
  \end{tikzpicture}
  \caption{Two semimodular lattices (left and center) and their
    vertical 2-sum, which is also semimodular (right).}
  \label{fig:v2sum}
\end{figure}
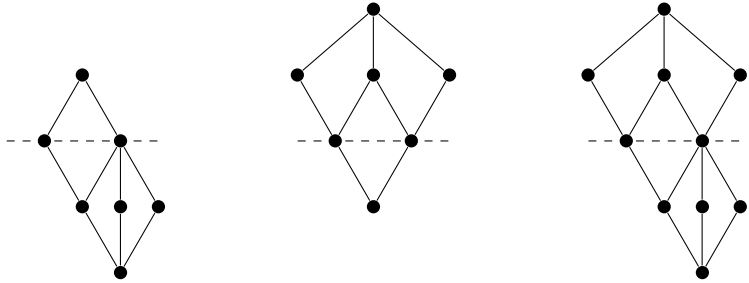
 
Let $L$ and $U$ be lattices such that $L$~has two coatoms and $U$~has
two atoms.  Then a \emph{vertical 2-sum} of $L$ and $U$ is a poset
obtained by removing $\top_L$ (the top of~$L$) and $\bot_U$ (the
bottom of~$U$), and identifying the coatoms of $L$ with the atoms
of~$U$.  The operation is illustrated in Figure~\ref{fig:v2sum}.

\begin{remark} The choice of which coatom is identified with which
  atom may give rise to two nonisomorphic vertical 2-sums, but we will
  not delve further into that issue here.  For our purposes it
  suffices that for any $L$ and $U$ there is at least one vertical
  2-sum, which we denote by $L +_2 U$, by a slight abuse of notation.
  Vertical 2-sums of several lattices can be defined in the obvious
  way by associativity.
  \label{remark:iso}
\end{remark}

\begin{remark}
  $L +_2 U$ has $|L|+|U|-4$ elements.
\end{remark}

\begin{lemma}
  A vertical 2-sum of two lattices is a lattice.
  \label{lemma:islattice}
\end{lemma}
\begin{proof}
  Let $V = L +_2 U$, and write for brevity $L' = L \setminus \top_L$
  and $U' = U \setminus \bot_U$.  Furthermore let $a,b$ be the two
  common elements of $L'$ and $U'$.  We claim that every pair of
  distinct elements $s,t \in V$ has a least upper bound.  We consider
  three cases.
  \begin{enumerate}
    
  \item Case $s,t \in U'$.  The claim holds because $U$ is a lattice.
    
  \item Case $s \in L'$ and $t \in U'$.  If $s \le t$, the claim is
    clear.  Otherwise, without loss of generality, let $s \le a
    \not\le t$ and $s \not\le b \le t$. Now $w = a \vee t$ is an upper
    bound of $s$ and $t$.  If $u \in U'$ is an upper bound of $s$ and
    $t$, we must have $a \le u$ (because $s \not\le b$), thus $u$ is
    an upper bound of $a \vee t = w$.  So $w$ is the least upper bound
    of $s$ and $t$ in~$V$.
    
  \item Case $s,t \in L'$.  Let $w$ be their least upper bound in~$L$.
    If $w = \top_L$, then $a \vee b$ is the least upper bound of $s,t$
    in~$V$.

    Now suppose $w \ne \top_L$, and let $u$ be any upper bound of $s$
    and $t$.  If $u \in L'$, we have $w \le u$ because $L$ is a
    lattice.  Let then $u \in U'$.  If $a,b \le u$, then $w \le u$.
    Otherwise, without loss of generality, let $a \le u$ and $b
    \not\le u$.  Because $s,t \le u$, we have $s,t \le a$.  Then $w
    \le a \le u$.  Thus $w$ is the least upper bound of $s,t$ in~$V$.
    
  \end{enumerate}
  We have shown that $V$ is a join-semilattice.  Since it has a bottom
  element $\bot = \bot_L$, it is also a lattice.
\end{proof}

\begin{remark}
  A vertical 2-sum of two vi-lattices is a vi-lattice, and a vertical
  2-sum of two graded lattices is graded.
\end{remark}

\begin{lemma}
  A vertical 2-sum of two semimodular lattices is semimodular.
  \label{lemma:issemimodular}
\end{lemma}
\begin{proof}
  Let $L,U$ be semimodular, $V = L +_2 U$, and $L' = L \setminus
  \top_L$ and $U' = U \setminus \bot_U$.  Let $s,t \in V$ such that
  $s,t \succ (s \wedge t)$.  Then either $s,t \in U'$ or $s,t \in L'
  \setminus U'$.  In the first case, $s,t \prec (s \vee t)$ because
  $U$ is semimodular.  In the second case, $s,t \prec (s \vee t)$
  because $L$ is semimodular.
\end{proof}

\begin{lemma}
  A vertical 2-sum of two modular lattices is modular.
  \label{lemma:ismodular}
\end{lemma}
\begin{proof}
  Apply Lemma~\ref{lemma:issemimodular} to both the vertical 2-sum and
  its dual.
\end{proof}

For families of graded vi-lattices, vertical 2-sum leads to a
recurrence analogous to Lemma~\ref{lemma:recursion}.  Let us first
define the building blocks that we are going to use.  If $X$ is a
graded lattice, we say that two elements of~$X$ are a \emph{neck} if
(1)~they have the same rank, (2)~they are the only elements having
that rank, and (3)~they are not atoms or coatoms.  We say that a
graded vi-lattice is a \emph{piece} if it has two atoms, two coatoms
and no neck, and its rank is at least three.  It follows that a piece
has at least six elements.

We can now state the recurrence.  For simplicity we state it as a
lower bound only; in particular, this implies that we need not
separate the cases where there are two nonisomorphic vertical 2-sums.

\begin{theorem}
  Let $\F$ be a family of graded vi-lattices that is closed under
  vertical 2-sum.  Let $\fvi(n)$ and $\fpc(n)$ be the numbers of
  $n$-element lattices and pieces in $\F$, respectively.  Let $N \ge 6$
  be an integer constant, and let $\flo \colon
  \natplus \to \natplus$ be the sequence defined by
  \[
  \flo(n) = \fpc(n)
  \]
  when $1 \le n \le 6$,
  \begin{equation}
    \flo(n) = \fpc(n) + \sum_{k=6}^{n-1} \fpc(k) \flo(n-k+4),
    \label{eq:twobase}
  \end{equation}
  when $7 \le n \le N$, and
  \begin{equation}
    \flo(n) = \sum_{k=6}^{N} \fpc(k) \flo(n-k+4)
    \label{eq:twotruncated}
  \end{equation}
  when $n \ge N+1$.  Then $\fvi(n) \ge \flo(n)$ for all $n \ge 1$.
  \label{thm:tworecursion}
\end{theorem}

\begin{proof}
  We prove by induction a stronger claim: that $\F$ contains at least
  $\flo(n)$ lattices of $n$ elements that have two coatoms and two
  atoms.  For $n \le 6$ the claim is clear.

  Let then $n \ge 7$.  For each $k$ such that $6 \le k \le n-1$, there
  are $\fpc(k)$ ways to choose a $k$-element piece~$L \in \F$ and, by
  the induction assumption, at least $\flo(n-k+4)$ ways to choose an
  $(n-k+4)$-element lattice $U \in \F$ that has two coatoms and two
  atoms.  For each choice of $L$ and~$U$, if $X = L +_2 U$, then $X$
  is in $\F$, has two coatoms and two atoms, and has $n$ elements.

  We claim that different choices of the pair $(L,U)$ cannot yield the
  same lattice $X$.  Suppose that $X = L +_2 U$, with $L$ and~$U$
  chosen as above.  Let $a$ and~$b$ be the neck of lowest rank in~$X$.
  The only way to represent $X$ as $X = L' +_2 U'$, so that $L'$ is a
  piece, is that $L' = L$ and $U' = U$.

  Adding up the choices, and including the $\fpc(n)$ pieces of $n$
  elements, we observe that in $\F$ there are at least
  \[
  \fpc(n) + \sum_{k=6}^{n-1} \fpc(k) \flo(n-k+4)
  \]
  $n$-element lattices that have two coatoms and two atoms.  For $n >
  N$ this can be further lower bounded by leaving out the first term
  and stopping the sum at $k=N$.  This concludes the induction.
\end{proof}

Since $\flo$ in Theorem~\ref{thm:tworecursion} is defined by a
homogeneous linear recurrence with constant terms, it can be lower
bounded by the same method as in the previous section, if $\fpc$ is
known up to $\fpc(N)$.  Modular vi-lattices of $n \le 30$ elements and
semimodular vi-lattices of $n \le 25$ elements were generated
in~\cite{kohonen2017}, and the listings are available in~\cite{eudat}.
With a short program we can check which of those vi-lattices are
pieces (as defined above), and count them.  From the counts we obtain
the following results.

\begin{proposition}
  $\mvi(n) \ge 2.1562^n$ for all $n$ large enough.
  \label{prop:mvi}
\end{proposition}

\begin{proof}
  The numbers of modular $n$-element pieces, for $n=6,7,\ldots,30$,
  are
  \begin{multline*}
    1,0,0,3,3, 4,15,27,52,117,259,554,1253,2802,6366,
    14429,33150, \\
    76090,175799,406851,946151,2204246,5153946,12076517,28375409.
  \end{multline*}
  Applying Theorem~\ref{thm:tworecursion} with these values, we obtain
  a sequence $\flo$ such that $\mvi(n) \ge \flo(n)$ for all $n \ge 1$.
  Numerically we find that the root of the auxiliary equation is a
  single real root $r \approx 2.156295$.  For a lower bound, we take
  $c = 2.1562$ and claim that $\flo(n) \ge c^n$ for $n \ge 150\,000$.
  This follows by induction as in the proof of
  Proposition~\ref{prop:m}.
\end{proof}

\begin{proposition}
  $\svi(n) \ge 2.6797^n$ for all $n$ large enough.
  \label{prop:svi}
\end{proposition}

\begin{proof}
  The numbers of semimodular $n$-element pieces, for $n=6,7,\ldots,25$,
  are
  \begin{multline*}
    1,0,0,5,6,9,40,122,323,964,2999,9374,30292,100539,339046,1159101,\\
    4018137,14116920,50263399,181341142.
  \end{multline*}
  Applying Theorem~\ref{thm:tworecursion} with these values, we obtain
  a sequence $\flo$ such that $\svi(n) \ge \flo(n)$ for all $n \ge 1$.
  Numerically we find that the root of the auxiliary equation is a
  single real root $r \approx 2.679797$.  For a lower bound, we take
  $c = 2.6797$ and claim that $\flo(n) \ge c^n$ for $n \ge 200\,000$.
  This follows by induction as in the proof of
  Proposition~\ref{prop:m}.
\end{proof}

\section{Concluding remarks}

This work was motivated by two empirical observations.  The first is
that modular vi-lattices are usually long and narrow
(cf.~\cite[Figs.~4 and~5]{kohonen2017}).  The second is that the
numbers of modular (vi-)lattices exhibit a rather stable exponential
growth, at least up to $n=30$.  Together these observations suggest
that much of that growth could be attributed to a ``Cartesian''
vertical combination of independently chosen parts.

In contrast, the vertical sum and 2-sum are not likely to be very
useful with lattice families whose members tend to be short and wide;
for example, with graded lattices, exponential bounds are superseded
by the already known bounds of the form $c^{n^{3/2}}$.

The notion of constructing vi-lattices by some kind of vertical
composition bears similarity to the work of Ern\'e \emph{et al.}  on
distributive lattices~\cite{erne2002}; however, their vertical
construction is different, and seems specific to distributive
lattices, as it works on finite posets that are in one-to-one
correspondence with finite distributive lattices (by a theorem of
Birkhoff).  Our vertical 2-sum works on lattices directly, and is
applicable to several lattice families.

It is tempting to extend the idea of the vertical 2-sum to lattices
that have more than two atoms and coatoms, but the result may not be a
lattice.  Consider, for example, defining \emph{vertical 3-sum}
($+_3$) as the obvious analogue of the vertical 2-sum.  Then the
analogue of Lemma~\ref{lemma:islattice} does not hold: for a
counterexample, if $B_3$ is the Boolean lattice of order $3$, then
$B_3 \: +_3 \: B_3$ is not a lattice.  In order to use such
generalized vertical sums for counting purposes, one needs an
efficient method of filtering out the non-lattices.  We leave such
studies for future research.


\bibliographystyle{plain}
\bibliography{refs}

\end{document}